\documentclass[10pt]{amsart}

\usepackage[utf8]{inputenc}

\usepackage{amsmath,amssymb,amsthm, epsfig}

\usepackage{hyperref}

\usepackage{esint}

\usepackage{amsmath}

\usepackage{amssymb}

\usepackage{hyperref}

\usepackage{color}

\usepackage{ulem}

\usepackage{epsfig}

\usepackage[mathscr]{eucal}

\usepackage{stackrel}

\usepackage{mathptmx} 
\usepackage[scaled=0.90]{helvet} 
\usepackage{courier} 
\normalfont
\usepackage[T1]{fontenc}

\usepackage[usenames,dvipsnames]{pstricks}

\usepackage{epsfig}

\usepackage{pst-grad} 

\usepackage{pst-plot} 

\title{Sharp regularity estimates for second order
fully nonlinear parabolic equations}
\author{João Vítor da Silva}
\address{Universidad de Buenos Aires, FCEyN, Department of Mathematics. Ciudad Universitaria-Pabell\'{o}n I -(C1428EGA), Buenos Aires - Argentina}
\email{jdasilva@dm.uba.ar}
\author{Eduardo V. Teixeira}
\address{Universidade Federal do Cear\'a,  Departmento de Matem\'{a}tica, Campus do Pici - Bloco 914, Fortaleza-CE - Brazil, CEP: 60455-760}
\email{teixeira@mat.ufc.br}

\date{}

\def \R {\mathbb{R}}

\def \esslimsup{\mathrm{ess limsup}}
\def \essliminf{\mathrm{ess liminf}}

\newcommand{\defeq}{\mathrel{\mathop:}=}

\newtheorem{theorem}{Theorem}[section]

\newtheorem{lemma}[theorem]{Lemma}

\newtheorem{proposition}[theorem]{Proposition}

\theoremstyle{definition}

\newtheorem{definition}[theorem]{Definition}

\theoremstyle{remark}

\newtheorem{remark}[theorem]{Remark}

\numberwithin{equation}{section}


\newcommand{\intav}[1]{\mathchoice {\mathop{\vrule width 6pt height 3 pt depth  -2.5pt
\kern -8pt \intop}\nolimits_{\kern -6pt#1}} {\mathop{\vrule width
5pt height 3  pt depth -2.6pt \kern -6pt \intop}\nolimits_{#1}}
{\mathop{\vrule width 5pt height 3 pt depth -2.6pt \kern -6pt
\intop}\nolimits_{#1}} {\mathop{\vrule width 5pt height 3 pt depth
-2.6pt \kern -6pt \intop}\nolimits_{#1}}}

\begin{document}
\maketitle

\begin{abstract}
We prove sharp regularity estimates for viscosity solutions of fully nonlinear parabolic equations of the form
\begin{equation}\label{Meq}\tag{Eq}
   u_t- F(D^2u, Du, X, t) = f(X,t) \quad \mbox{in} \quad Q_1,
\end{equation}
where $F$ is elliptic with respect to the Hessian argument and $f \in L^{p,q}(Q_1)$. The quantity $\kappa(n, p, q):=\frac{n}{p}+\frac{2}{q}$ determines to which regularity regime a solution of \eqref{Meq} belongs. We prove that when $1< \kappa(n,p,q) < 2-\epsilon_F$, solutions are parabolic-H\"{o}lder continuous for a sharp, quantitative exponent $0< \alpha(n,p,q) < 1$. Precisely at the  critical borderline case, $\kappa(n,p,q)= 1$, we obtain sharp Log-Lipschitz regularity estimates. When $0< \kappa(n,p,q) <1$, solutions are locally of class $C^{1+ \sigma, \frac{1+ \sigma}{2}}$ and in the limiting case $\kappa(n,p,q) = 0$, we show $C^{1, \text{Log-Lip}}$ regularity estimates provided $F$ has ``better'' \textit{a priori} estimates.



\end{abstract}
\section{Introduction}\label{Introduction}

\hspace{0.65cm}The study of second order parabolic equations plays a fundamental role in the development of several fields in pure and applied mathematics, such as differential geometry, functional and harmonic analysis, infinite dimensional dynamical systems, probability, as well as in mechanics, thermodynamics, electromagnetism, among others. The non-homogeneous heat equation,
\begin{equation}\label{prot}
u_t - \Delta u = f \quad \mbox{in} \quad Q_1 = B_1 \times (-1,0],
\end{equation}
$f \in L^p(Q_1)$, $p> \frac{n+2}{2}$, represents the simplest linear prototype. Its mathematical analysis goes back to 19th century and the regularity theory for such an equation is nowadays fairly complete.  The fully nonlinear parabolic theory is quite more recent. The fundamental works of Krylov and Safonov, \cite{KS1}, \cite{KS2} on linear, non-divergence form elliptic equations set the beginning of the development of the regularity theory for viscosity solutions to fully nonlinear parabolic equations. Since then this has been a central subject of research. Wang in \cite{Wang 1, Wang 2} proves Harnack inequality and $C^{1+ \alpha,\frac{1+ \alpha}{2}}$ estimates for fully nonlinear parabolic equations, and Crandall \textit{et al} in \cite{CKS} develop an $L^p$-viscosity theory. Krylov in  \cite{KRYL1, KRYL2} obtains $C^{2+\alpha, \frac{2+\alpha}{2}}$ estimates for solutions to $u_t-F(D^2u) = 0$, under convexity assumptions,  and Caffarelli and Stefanelli in \cite{CS} exhibit solutions to uniform parabolic equations that are not $C^{2, 1}.$

Non-divergence form parabolic equations involving sources with mixed integrability conditions $f \in L^{p,q}(Q_1)$, as in \eqref{Meq} have also been fairly well studied in the literature. Existence in suitable parabolic Sobolev spaces has been proven by Krylov,  see \cite{KRYL4, KRYL5}, see also the sequence of works by Kim \cite{Kim1, Kim3}. Insofar as regularity estimates are concerned, only qualitative results are available when $p$ and $q$ are sufficient large. Nonetheless, as in a number of physical, geometric and free boundary problems, obtaining a quantitative sharp regularity estimate for solutions is decisive for a refine analysis. Hence, the purpose this paper is to obtain sharp moduli of continuity to solutions for second order parabolic equation \eqref{Meq}, involving sources with mixed norms, which depends only on  dimension, $p$ and $q$.

Hereafter we denote by
$$
	\kappa = \kappa(n,p,q) \defeq  \frac{n}{p} + \frac{2}{q}.
$$
The first quantitative regularity result we show states that if  $1< \kappa(n,p,q) < \frac{n+2}{n_P},$  where $\frac{n+2}{2} < n_P < n+1$ is universal, 
then solutions are parabolically $\alpha$-H\"older continuous for the sharp exponent
$ \alpha \defeq	 2-\left(\frac{n}{p} + \frac{2}{q}\right)$.

Intuitively, as $\kappa(n,p,q)$ decreases, one should expect that regularity estimates of solutions improve	. The borderline is $\kappa(n,p,q) = 1$, where we prove that solutions are parabolically Log-Lipschitz continuous. The result is a further quantitative improvement to the fact that $u \in C_{loc}^{\alpha, \frac{\alpha}{2}}(Q_1)$ for any $0<\alpha <1$.

When $0< \kappa(n,p,q) <1$, we show that solutions are $C^{1+\beta, \frac{1+\beta}{2}}$, for $\beta \lesssim 1-\left(\frac{n}{p} + \frac{2}{q}\right)$. Qualitative results, when $p=q > n+1$, were previously obtained by Crandall \textit{et al} \cite{CKS} and Wang \cite{Wang 2}.

Finally, we deal with the upper borderline case, $f\in \text{BMO}(Q_1)$. Under appropriate higher \textit{a priori} estimates on $F$, we show that solutions are $C_{loc}^{1, \text{Log-Lip}}(Q_1)$. Particularly, $u \in C_{loc}^{1+\alpha, \frac{1+\alpha}{2}}(Q_1)$ for any $0<\alpha <1$.

The table below provides a global picture of the parabolic regularity theory for equations with anisotropic sources, in comparison with the sharp elliptic estimate from \cite{Tei}:

\begin{table}[h]
\centering
\begin{tabular}{|c|c|c|c|c}
\cline{1-4}
 $f \in L^p(B_1)$  & Regularity of $u$ & $f \in L^{p,q}(Q_1)$ & Regularity of $u$  \\ \cline{1-4}
   $n-\varepsilon \leq p<n$ &  $C_{loc}^{0,2-\frac{n}{p}}(B_1)$ & $1< \frac{n}{p} + \frac{2}{q}< \frac{n+2}{n_P}$ & $C_{loc}^{\varsigma, \frac{\varsigma}{2}}(Q_1)$ \\ \cline{1-4}
   $p=n$ &  $C_{loc}^{0,\textrm{Log-Lip}}(B_1)$ & $\frac{n}{p} + \frac{2}{q}=1$ & $C_{loc}^{0,\textrm{Log-Lip}}(Q_1)$ \\ \cline{1-4}
   $p>n$ &  $C_{loc}^{1,\min\left\{\alpha^{-}, 1-\frac{n}{p}\right\}}(B_1)$ & $0<\frac{n}{p} + \frac{2}{q}<1$& $C_{loc}^{1+\mu, \frac{1+\mu}{2}}(Q_1)$  \\ \cline{1-4}
   $ \text{BMO} \varsupsetneq L^{\infty}$ & $C_{loc}^{1, \textrm{Log-Lip}}(B_1)$ & $\text{BMO} \varsupsetneq L^{\infty}$& $C_{loc}^{1, \textrm{Log-Lip}}(Q_1)$ \\\cline{1-4}
\multicolumn{4}{|l|}{ \hspace{1.34cm} {\bf Elliptic Theory  \quad \quad \quad \quad X \quad \quad \quad \quad  Parabolic Theory}} &  \\ \cline{1-4}
\end{tabular}
\end{table}
where $\varsigma \defeq 2-\left(\frac{n}{p} + \frac{2}{q}\right)$ and $\mu \defeq \min\left\{\alpha^{-}, 1-\left(\frac{n}{p} + \frac{2}{q}\right)\right\}$, and $\alpha^{-}$ means $\alpha-\epsilon$ for every $0< \epsilon$.  	

It is interesting to note that the parabolic regularity estimates agree with its elliptic counterpart provided $f \in L^{p, \infty}(Q_1)$.


Next picture shows the critical surfaces and the regions they define for the optimal regularity estimates available for solutions to \eqref{Meq}.

\pagebreak

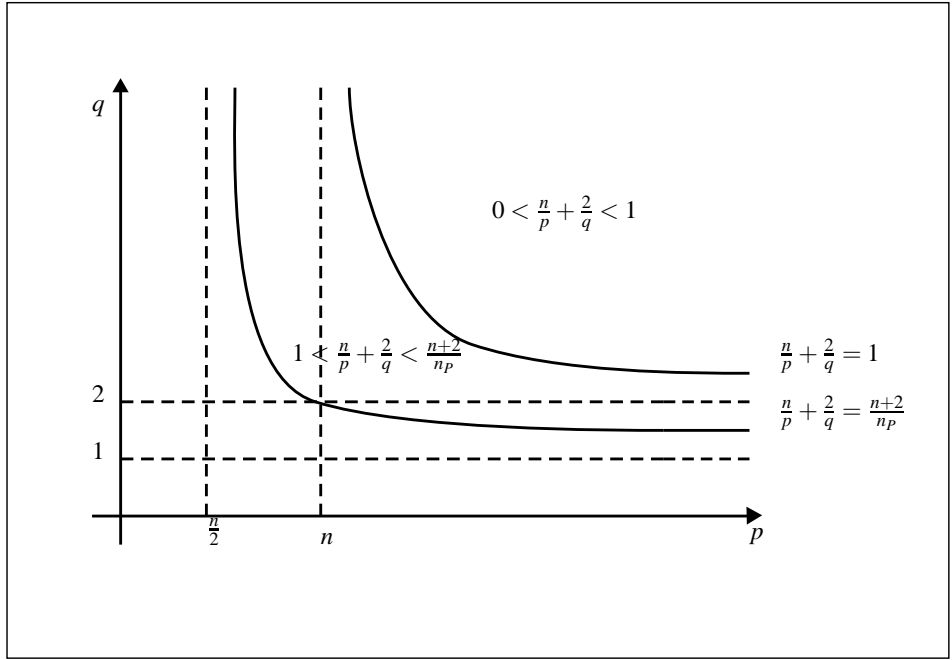
\begin{figure}[h]
\begin{center}
\psscalebox{0.95} 
{
\begin{pspicture}(0,-4.6)(13.2,4.6)
\psframe[linecolor=black, linewidth=0.02, dimen=outer](13.2,4.6)(0.0,-4.6)
\psline[linecolor=black, linewidth=0.04](1.6,3.4)(1.6,-3.0)
\psline[linecolor=black, linewidth=0.04](1.2,-2.6)(10.4,-2.6)
\psline[linecolor=black, linewidth=0.04, linestyle=dashed, dash=0.17638889cm 0.10583334cm](2.8,-2.6)(2.8,3.4)
\psline[linecolor=black, linewidth=0.04, linestyle=dashed, dash=0.17638889cm 0.10583334cm](1.6,-1.8)(9.2,-1.8)
\psline[linecolor=black, linewidth=0.04, linestyle=dashed, dash=0.17638889cm 0.10583334cm](1.6,-1.0)(9.2,-1.0)
\psbezier[linecolor=black, linewidth=0.04](3.2,3.4)(3.2,2.6)(3.05,-0.6)(4.325,-1.0)(5.6,-1.4)(8.45,-1.4)(9.2,-1.4)
\psbezier[linecolor=black, linewidth=0.04](4.8,3.4)(4.8,2.6)(5.309091,0.2)(6.523077,-0.2)(7.737063,-0.6)(9.538462,-0.6)(10.4,-0.6)
\psline[linecolor=black, linewidth=0.04, linestyle=dashed, dash=0.17638889cm 0.10583334cm](4.4,3.4)(4.4,-2.6)
\rput[bl](10.4,-3.0){$p$}
\rput[bl](1.2,3.0){$q$}
\rput[bl](4.4,-3.0){$n$}
\rput[bl](2.8,-3.0){$\frac{n}{2}$}
\rput[bl](6.8,1.4){$0<\frac{n}{p}+\frac{2}{q}<1$}
\rput[bl](10.8,-0.6){$\frac{n}{p}+\frac{2}{q}=1$}
\rput[bl](1.2,-1.8){$1$}
\rput[bl](1.2,-1.0){$2$}
\rput[bl](4.0,-0.6){$1<\frac{n}{p}+\frac{2}{q}< \frac{n+2}{n_P}$}
\psdots[linecolor=black, dotstyle=triangle*, dotsize=0.2](1.6,3.4)
\rput{31.265638}(0.19308585,-5.7662797){\psdots[linecolor=black, dotstyle=triangle*, dotsize=0.22](10.4,-2.6)}
\psline[linecolor=black, linewidth=0.04](9.2,-1.4)(10.4,-1.4)
\psline[linecolor=black, linewidth=0.04, linestyle=dashed, dash=0.17638889cm 0.10583334cm](9.2,-1.0)(10.4,-1.0)
\psline[linecolor=black, linewidth=0.04, linestyle=dashed, dash=0.17638889cm 0.10583334cm](9.2,-1.8)(10.4,-1.8)
\rput[bl](10.8,-1.4){$\frac{n}{p}+\frac{2}{q}=\frac{n+2}{n_P}$}
\end{pspicture}
}
\end{center}
\caption{ Critical surfaces for optimal regularity estimates.}
\end{figure}

\noindent  {\bf Acknowledgement.} This article is part of the first author's Ph.D thesis. He would like to thank the Department of Mathematics at Universidade Federal do Cear\'a  for fostering a pleasant and productive scientific atmosphere, which has benefited a lot the final outcome of this current project. This work has been partially supported by Capes and CNPq, Brazil.

\section{Definitions and preliminary results} \label{PT}

\hspace{0.65cm}Throughout this paper $F \colon Sym(n) \times \R^n \times B_1(0) \times (-1,0]\longrightarrow \mathbb{R}$ is a fully nonlinear uniformly elliptic operator with respect to the Hessian argument and Lipschitz with respect to gradient dependence. That is, there are constants $\Lambda \ge \lambda >0$ and $\upsilon \geq 0$ such that for all $Z, W \in R^n$ and $M,N \in \mbox{Sym}(n)$, space of $n \times n$ symmetric matrices, with $M \geq N$, there holds
\begin{equation}\label{eq1}
    \mathcal{M}^{-}_{\lambda,\Lambda}(M-N)- \upsilon |Z-W| \le F(M, Z, X, t)-F(N, W,  X, t) \le \mathcal{M}^{+}_{\lambda,\Lambda}(M-N) + \upsilon |Z-W|.
\end{equation}
Hereafter, $\mathcal{M}^{\pm}_{\lambda,\Lambda}$ denote the Pucci's extremal operators:
$$
    \mathcal{M}^{+}_{\lambda,\Lambda}(M) := \lambda \cdot \sum_{e_i <0} e_i  + \Lambda \cdot \sum_{e_i >0} e_i \quad \mbox{and} \quad \mathcal{M}^{-}_{\lambda,\Lambda}(M) := \lambda \cdot \sum_{e_i >0} e_i  + \Lambda \cdot \sum_{e_i <0} e_i
$$
where $\{e_i : 1 \le i \le n\}$ are the eigenvalues of $M$. We can (and will) always assume that $F$ is normalized as $F(0, 0, X, t) = 0$. Any operator $F$ which satisfies the condition
\eqref{eq1} will be referred in this article as a \textit{$(\lambda,
\Lambda, \upsilon)$-parabolic operator}. Following classical terminology, any constant or mathematical term which depends only on dimension and of the parabolic parameters
$\lambda$, $\Lambda$ and $\upsilon$ will be called \textit{universal}.

%
%
%
%
%

Equations and problems studied here are designed in the $(n+1)$-dimensional Euclidean space, $\mathbb{R}^{n+1}$. The semi-open cylinder is denoted by
$Q_r(X_0, \tau) = B_r(X_0) \times (\tau -r^2, \tau]$. For simplicity we refer $Q_1(0, 0) = Q_1$. The \textit{parabolic distance} between the points $P_1=(X_1, t_1)$ and $P_2=(X_2, t_2) \in Q_1$ is defined by
$$
   \displaystyle d_{\text{par}}(P_1, P_2) \defeq \sqrt{|X_1-X_2|^2 + |t_1-t_2|}.
$$

For a function $u\colon Q_1 \rightarrow \mathbb{R}$ the semi-norm and norm for the {\it parabolic H\"{o}lder space} are defined respectively by

\begin{equation}\label{HNorm}
   \displaystyle [u]_{C^{\alpha, \frac{\alpha}{2}}(Q_1)} \defeq \sup_{(X, t),(Y, s)\in Q_1 \atop{(X, t) \not= (Y, s)}} \frac{|u(X, t)-u(Y, s)|}{d_{\text{par}}((X, t),(Y,s))^{\alpha}} \quad \text{and} \quad \|u\|_{C^{\alpha, \frac{\alpha}{2}}(Q_1)} \defeq \|u\|_{C^{0}(Q_1)} + [u]_{C^{\alpha, \frac{\alpha}{2}}(Q_1)}.
\end{equation}
Under finiteness of such a norm one concludes that $u$ is $\alpha$-H\"{o}der continuous with respect to the spatial variables and $\frac{\alpha}{2}-$H\"{o}der with respect to the temporal variable. $C^{1+\alpha, \frac{1+\alpha}{2}}(Q_1)$ is the space of $u$ whose spacial gradient $Du(X,t)$ there exists in the classical sense for every $(X, t) \in Q_1$ and such that
$$
	\begin{array}{lll}
  		  \|u\|_{C^{1+ \alpha, \frac{1+ \alpha}{2}}(Q_1)} &:=& \|u\|_{L^{\infty}(Q_1)} + \|Du\|_{L^{\infty}(Q_1)} \\
		 	&+& \displaystyle \sup_{(X, t),(Y, s)\in Q_1 \atop{(X, t) \not= (Y, s)}} \frac{|u(X,t)-u(Y, \tau)-D u(X,t)\cdot (X-Y)|}{d_{\text{par}}^{1+\alpha}((X,t),(Y, s))}
	\end{array}
$$
is finite. It is easy to verify that $u \in C^{1+ \alpha, \frac{1+ \alpha}{2}}(Q_1)$ implies every component of $Du$ is $C^{0, \alpha}(Q_1)$ , and $u$ is $\frac{1+\alpha}{2}-$H\"{o}lder continuous in the variable $t$, see for instance \cite{CKS}.

A function  $u$ belongs to the Sobolev space $W^{2,1,p}(Q_1)$ if it satisfies $u, D u, D^2 u, u_t \in L^p(Q_1)$. The corresponding norm is given by
$$
    \|u\|_{W^{2,1,p}(Q_1)} = [\|u\|^p_{L^p{(Q_1)}}+\|u_t\|_{L^{p}(Q_1)}^p+\|D u\|_{L^{p}(Q_1)}^p+ \|D^2 u\|_{L^{p}(Q_1)}^p]^{\frac{1}{p}}
$$
It follows by Sobolev embedding that if $p>\frac{n+2}{2}$ then $W^{2,1,p}(Q_1)$ is continuously embedded in $C^0(Q_1)$. Also, $u \in W_{loc}^{2,1,p}(Q_1)$ implies that $u$ is twice parabolically differentiable a.e., see for more details \cite{CKS}.

\begin{definition}[$L^P$-viscosity solutions]\label{VS}Let $G \colon Sym(n) \times \R^n \times \R \times B_1(0) \times (-1,0]\to \mathbb{R}$ be a uniformly elliptic operator, $P> \frac{n+2}{2}$ and $f \in L_{loc}^{P}(Q_1)$. We say that a function $u \in C^0(Q_1)$ is an $L^{P}$-viscosity subsolution (respectively supersolution) to

\begin{equation}\label{eqVS}
u_t-G(D^2 u(X,t), Du(x, t), u(X,t), X, t) = f(X,t) \quad in \quad Q_1
\end{equation}
if for all $\varphi \in W_{loc}^{2, 1 ,P} (Q_1)$ whenever $\varepsilon >0$ and $\mathcal{O} \subset Q_1$ is an open and
$$
\varphi_t-G(D^2 \varphi(X,t), D\varphi(x, t), \varphi(X,t), X, t) - f(X,t) \geq  \varepsilon \quad (resp. \leq - \varepsilon)\quad a.e. \quad in \quad \mathcal{O}
$$
then $u-\varphi$ cannot attains a local maximum (resp. minimum) in $\mathcal{O}$. In an equivalent manner, $u$ is an $L^{P}-$viscosity subsolution (resp. supersolution) if for all test function $\varphi \in W_{loc}^{1, 2 ,P} (Q_1)$ and $(X_0, t_0) \in Q_1$ at which $u - \varphi$ attain a local maximum (resp. minimum) one has

$$
   \displaystyle \stackrel[(X, t) \to (X_0, t_0)]{}{\essliminf} [\varphi_t-G(D^2 \varphi(X,t), D \varphi(x, t), \varphi(X,t), X, t) - f(X,t)] \leq 0
$$

$$
    \displaystyle \stackrel[(X, t) \to (X_0, t_0)]{}{\esslimsup}  [\varphi_t-G(D^2 \varphi(X,t), D \varphi(x, t), \varphi(X,t), X, t) - f(X,t)] \geq 0
$$
Finally we say that $u$ is an $L^{P}$-viscosity solution to \eqref{eqVS} if it is both an $L^{P}$-viscosity supersolution and an $L^{P}$-viscosity subsolution.
\end{definition}

According to \cite{CKS} and \cite{Wang 1} for a fixed $(X_0, \tau) \in Q_1$, we
measure the oscillation of the coefficients of $F$ around $(X_0, \tau )$ by the quantity
\begin{equation}\label{osc}
   \displaystyle \Theta_F(X_0,\tau,  X, t) := \sup_{M \in Sym(n)} \frac{|F(M, 0, X, t)-F(M, 0, X_0, \tau)|}{\|M\| + 1}.
\end{equation}
For notation purposes, we shall often write $\Theta_F(0, 0, X, t) = \Theta_F(X, t)$.

We recall that a function $f$ is said to belong to the anisotropic Lebesgue space, $L^{p,q}(Q_1)$ if
$$
    \displaystyle \|f\|_{L^{p, q}(Q_1)} := \left(\int_{-1}^{0}  \left( \int_{B_1}|f(X,t)|^pdX\right)^{\frac{q}{p}}dt\right)^{\frac{1}{q}} = \|\|f( \cdot, t )\|_{L^p(B_1)}\|_{L^q((-1,0])} < +\infty.
$$
This is a Banach space when endowed with the norm above.  When $p=q$, this is the standard definition of $L^p$ spaces. The definition are naturally extended when either $p$ or $q$ are infinity. It is plain to verify that $L^{p,q}(Q_1) \subset L^s(Q_1)$ for $s \defeq \min\{p, q\}$.

We recall the existence of a constant $n_P$, satisfying $\frac{n+2}{2}\leq n_P<n+1$, for which Harnack inequality holds for $L^P$-viscosity solutions, provided  $P>n_P$, see for instance \cite{CKS}.  The following compactness result becomes then available:

\begin{proposition}[Compactness of solutions]\label{CompSol} Let $u \in C^0(Q_1)$ be an $L^P$-viscosity solution to \eqref{Meq} under the assumption  $P \geq \min\{p, q\} > n_P$. Then
$u$ is locally of class  $C^{\beta, \frac{\beta}{2}}$ for some $0<\beta<1$ and
$$
   \displaystyle \|u\|_{C^{\beta, \frac{\beta}{2}}(Q_r)} \leq C(n, \lambda, \Lambda)r^{-\beta}\left(\|u\|_{L^{\infty}(Q_{2r})}+ r^{2-\left(\frac{n}{p}+\frac{2}{q}\right)}\|f\|_{L^{p, q}(Q_{2r})}\right).
$$
\end{proposition}

In the sequel, we obtain a Lemma which provides a tangential path toward the regularity theory available for constant coefficient, homogeneous $F$-caloric functions.

\begin{lemma}[$F$-caloric approximation Lemma]\label{Lemma1} Let $u \in C^{0}(Q_1)$ be an $L^P$-viscosity solution to \eqref{Meq} with $|u|\leq 1$ and $f\in L^{p,q}(Q_1)$ with $P \defeq \min\{p,q\} >  n_P$. Given $\delta >0$, there exists $\eta = \eta(n, \Lambda, \lambda, \delta) >0$ such that if

$$
   \displaystyle \max\left\{\left(\intav{Q_1} \Theta^{P}_F(X, t)\right)^{\frac{1}{P}}, \|f\|_{L^{p, q}(Q_1)}, \upsilon\right\} \leq \eta,
$$
then we can find a function $h \colon Q_{1/2}\to
\mathbb{R}$ and a $(\lambda, \Lambda, 0)-$parabolic, constant
coefficients operator $\mathfrak{F}\colon  Sym(n)\to
\mathbb{R},$ such that

\begin{equation}\label{eq2}
    h_t-\mathfrak{F}(D^2 h)=0, \quad \mbox{in} \quad Q_{1/2}
\end{equation}
in the $L^P$-viscosity sense, and, moreover

\begin{equation}\label{eq3}
    \displaystyle \sup_{ (X, t) \in Q_{1/2}} |(u-h)(X,t)| \leq \delta.
\end{equation}
\end{lemma}
\begin{proof} The proof is based on a contradiction argument. Suppose that there exists a ${\delta}_0>0$ for which the thesis of Lemma \ref{eq3} is not true. That is, we could find a sequence of functions $(u_j)_{j\geq 1}$, with $|u_j|\leq 1$ in $Q_1$, a sequence
of $(\lambda, \Lambda, \upsilon_j)$-operators $F_j\colon Sym(n) \times \R^n \times Q_1\to \mathbb{R}$ and a sequence of functions $(f_j)_{j \geq 1}$ satisfying
\begin{equation}\label{eq4}
   (u_j)_t - F_j(D^2 u_j, Du_j, X, t)= f_j(X, t) \quad \mbox{in} \quad Q_1
\end{equation}
in the $L^P$-viscosity sense, with
\begin{equation}\label{eq5}
    \displaystyle \max\left\{\left(\intav{Q_1} \Theta^{P}_{F_j}(X, t)\right)^{\frac{1}{P}},  \|f_j\|_{L^{p, q}(Q_1)}, \upsilon_j\right\}= \text{o}(1) \quad as \quad j \to \infty,
\end{equation}
however
\begin{equation}\label{eq6}
    \displaystyle \sup_{ (X, t) \in Q_{1/2}} |(u_j-h)(X, t)| > {\delta}_0
\end{equation}
for all $h$ which satisfies \eqref{eq2} and all $(\lambda, \Lambda, 0)$-operator $\mathfrak{F}$. By H\"{o}lder regularity of the sequence $(u_j)_{j\geq 1}$, Proposition \ref{CompSol}, we
may assume, passing to a subsequence if necessary, that $u_j \to u_0$ locally uniformly in $Q_1$. Furthermore, it follows from structural condition \eqref{eq1} of the sequence of operators $(F_j)_{j\geq 1}$ that $F_j(M, Z, X, t)\to F_0(M, Z, X, t)$ locally uniformly in the space $Sym(n)\times \R^n$ for each $(X, t) \in Q_1$ fixed. Moreover, by hypothesis \eqref{eq5}, $F_0$ is a $(\lambda, \Lambda, 0)$ constant coefficients operator, see for instance \cite{CKS} and \cite{Wang 1}. To conclude the proof, we use stability arguments, see \cite[Section 6]{CKS}, as to deduce that
$$
   (u_0)_t-F_0(D^2 u_0)=0 \quad \mbox{in} \quad Q_{1/2},
$$
in the $L^P$-viscosity sense, This gives a contradiction to \eqref{eq6} to $j\gg 1$ and the proof of the Lemma is concluded.
\end{proof}


%

We conclude this section by commenting on reduction processes to be used throughout the proof.

\begin{remark}\label{Rem1}[Preserving ellipticity] If $F$ is a $(\lambda, \Lambda, \upsilon)$-parabolic operator then
$$
	G(M, \overrightarrow{Z}, X, t)= \kappa^2F\left(\frac{M}{\kappa^2}, \frac{\overrightarrow{Z}}{\kappa} X, t\right)
$$
 is a $(\lambda,\Lambda, \kappa \upsilon)$-parabolic operator.
\end{remark}

\begin{remark}\label{Rem2} [Normalization and scaling] We can always suppose, without loss of generality, that viscosity solutions of
$$
u_t-F(D^2u, Du, X, t) = f(X,t) \quad in \quad Q_1
$$
satisfy $\|u\|_{L^{\infty}(Q_1)} \leq 1$. Also given a small number $\varepsilon_0>0$, we can also suppose that $\upsilon  + \|f\|_{L^{p,q}(Q_1)} < 2\varepsilon_0$. In fact, for $\kappa \defeq \frac{\varepsilon_0}{\varepsilon_0\|u\|_{L^{\infty}(Q_1)}+ \|f\|_{L^{p,q}(Q_1)}}$ and $R>\max\left\{1, \frac{\varepsilon_0}{\upsilon}\right\}$, defining
$$
	v(X, t) \defeq \kappa u(R^{-1}X,R^{-2}t)
$$
we readily verify that
\begin{enumerate}
\item $\|v\|_{L^{\infty}(Q_1)} \leq 1$;
\item $v_t - G(D^2v, Dv, X, t) = g(X, t)  \mbox{ in }  Q_1,$ in the $L^P$-viscosity sense, where
$$
G(M, \overrightarrow{Z}, X, t) = \frac{\kappa}{R^2}F\left(\frac{R^2}{\kappa}M, \frac{R}{\kappa}\overrightarrow{Z}, R^{-1}X, R^{-2}t\right) \quad and \quad g(X,t) = \frac{\kappa}{R^2}f(R^{-1}X, R^{-2}t);
$$
\item $G$ is a $(\lambda, \Lambda, \overline{\upsilon})$-parabolic operator, with $\overline{\upsilon}<\varepsilon_0$;
\item $\|g\|_{L^{p,q}(Q_1)} \leq \varepsilon_0$.
\end{enumerate}

\end{remark}


\section{Optimal $C^{\alpha, \frac{\alpha}{2}}$ regularity}\label{HR}

\hspace{0.65cm}Our strategy for proving optimal $C^{\alpha, \frac{\alpha}{2}}$ regularity estimates is based on a refined compactness method as in \cite{CKS, Tei, Wang 1, Wang 2}. It relies on a control of decay of oscillation based on the regularity theory available for a nice limiting equation.  Roughly speaking the \textit{geometric tangential analysis} of the limit arising from of family of fully nonlinear parabolic operators $F_i$ as we are in smallest regime on the source term and on oscillation of coefficients of the respective operators. Next lemma is the key access point for the approach, as it provides the first step in the iteration process to be implemented.

\begin{lemma}\label{lemma3.1} Let $u \in C^{0}(Q_1)$ be a normalized $L^P$-viscosity solution for \eqref{Meq}, that is, $|u|\leq 1$ in $Q_1$. Given $0<\gamma<1$, there exist $\eta(\Lambda, \lambda, n, \gamma) >0$ and $0<\rho(\Lambda, \lambda, n, \gamma)\ll\frac{1}{2}$, such that if

$$
   \displaystyle \max\left\{\left(\intav{Q_1} \Theta^{P}_F(X, t)\right)^{\frac{1}{P}},  \|f\|_{L^{p, q}(Q_1)}, \upsilon\right\} \leq \eta \quad  \mbox{with} \quad 1<\frac{n}{p} + \frac{2}{q}< \frac{n+2}{n_P}
$$
then, for some $\varsigma \in \mathbb{R}$, with $|\varsigma| \leq C(\Lambda, \lambda,n)$ there holds
\begin{equation}\label{eq3.1}
    \displaystyle \sup_{Q_{\rho}} |u - \varsigma| \leq \rho^{\gamma}.
\end{equation}
\end{lemma}

\begin{proof} For a $\delta > 0$ to be chosen \textit{a posteriori}, let $h$ be a solution to a homogeneous uniformly parabolic equation with constant coefficients, that is $\delta$-close to $u$ in the $L^{\infty}$-norm, i.e.,
\begin{equation}\label{eq3.2}
   \displaystyle h_t - \mathfrak{F}(D^2 h)=0 \quad \mbox{in} \quad Q_1 \quad \text{and} \quad \sup_{Q_{1/2}}|(u-h)(X, t)|\leq \delta.
\end{equation}
Lemma \ref{Lemma1} assures the existence of such a
function. Once our choice for $\delta$, to be set of the end of this proof, is
universal, then the choice of $\eta(n, \lambda, \Lambda, \delta)$ is too universal.
From the regularity theory available for $h$,  see for instance \cite{CKS} or \cite{Wang 2}, we can estimate

\begin{equation}\label{eq3.3}
    \displaystyle  |h(X, t)-h(0, 0)|\leq C(n, \lambda, \Lambda) d_{\text{par}}((X,t),(0,0)) \quad \forall \quad |X|^2+|t| < \frac{1}{3},
\end{equation}
and also,
\begin{equation}\label{eq3.4}
   |h(0,0)| \leq C.
\end{equation}
For $\varsigma = h(0, 0)$ it follows from equations \eqref{eq3.2} and \eqref{eq3.3} via triangular
inequality that
\begin{equation}\label{eq3.5}
    \displaystyle \sup_{Q_{\rho}} |u(X, t)- \varsigma|\leq \delta + C(n, \lambda, \Lambda) \rho.
\end{equation}
We make the following universal selections:
\begin{equation}\label{choose}
   \displaystyle \rho \defeq \left(\frac{1}{2C}\right)^{\frac{1}{1- \gamma}} \quad and  \quad  \delta  \defeq \frac{1}{2}\rho^{\gamma}
\end{equation}
where $C>0$ is a universal constant from equation \eqref{eq3.3}. Let us stress that the choices above depend only upon dimension, ellipticity parameters and the fixed exponent. From the above choices we obtain
$$
  \begin{array}{ll}
     \displaystyle \sup_{Q_{\rho}} |u(X, t)-\varsigma| & \leq {\rho}^{\gamma} \\
  \end{array}.
$$
and the Lemma is concluded.	
\end{proof}

\begin{theorem}\label{thm3.2} Let $u \in C^{0}(Q_1)$ be an $L^P$-viscosity solution of \eqref{Meq} with $f \in L^{p, q}(Q_1)$ with
$$
	1<\frac{n}{p} + \frac{2}{q}< \frac{n+2}{n_P}.
$$
There exists a universal constant $\theta_0>0$ such that if
$$
	\displaystyle \sup_{(Y, \tau) \in Q_{1/2}}  \left(\intav{Q_1} \Theta^P_F(Y, \tau, X, t)\right)^{\frac{1}{P}} \leq \theta_0,
$$
then, for a universal constant $C>0$ and $\alpha \defeq 2-\left(\frac{n}{p} + \frac{2}{q}\right)$, there holds

$$
   \displaystyle \|u\|_{C^{\alpha, \frac{\alpha}{2}}(Q_{1/2})} \leq C\{\|u\|_{L^{\infty}(Q_1)}+ \|f\|_{L^{p, q}(Q_1)}\}.
$$

\end{theorem}

\begin{proof} Through normalization and scaling processes, see Remark \ref{Rem2}, we can suppose without losing generality that $|u|\leq 1$ and $\|f\|_{L^{p, q}(Q_1)} \leq \eta$, where $\eta$ is the universal constant from Lemma \ref{lemma3.1} when we set $\gamma = \xi(n , p, q) = 2-\left(\frac{n}{p} + \frac{2}{q}\right)$. Once selected $\theta_0  = \eta$ the goal will be to iterate the Lemma \ref{lemma3.1}. For a fixed $(Y, \tau) \in Q_{1/2}$ we claim that there exists a convergent sequence of real numbers $\{\varsigma_k\}_{k \geq 1}$, such that
\begin{equation}\label{eq3.7}
    \displaystyle \sup_{Q_{\rho^k}(Y, \tau)} |u(X, t) - \varsigma_k| \leq \rho^{k\xi(n , p, q)}
\end{equation}
where the radius $0<\rho \ll \frac{1}{2}$ is given by Lemma \ref{lemma3.1}, upon the selection of $\gamma$ as above.

The proof of \eqref{eq3.7} will follow by induction process. Lemma\ref{lemma3.1} gives the first step of induction, $k = 1$.  Now suppose verified the $k^{th}$ step in \eqref{eq3.7}. We define
$$
    \displaystyle v_k(X, t) =  \frac{u(Y + \rho^k X, \tau + \rho^{2k}t)}{\rho^{k \xi(n , p, q)}}
$$
and
$$
   F_k(M, Z,  X, t) \defeq \rho^{k \left[2- \xi(n , p, q)\right]} F\left(\frac{1}{ \rho^{k \left[2- \xi(n , p, q)\right]}}M, \frac{1}{\rho^{k \left[1- \xi(n , p, q)\right]}}Z, Y+ \rho^kX, \tau+ \rho^{2k}t\right).
$$
As commented before, see Remark \ref{Rem1}, $F_k$ is $(\lambda, \Lambda, \upsilon)$-parabolic operator, moreover  by the induction hypothesis, $|v_k|\leq 1$and
$$
   (v_k)_t - F_k(D^2v_k, Dv_k, X,t) = \rho^{k. \left[2- \xi(n , p, q)\right]} f(Y + \rho^k X,  \tau + \rho^{2k}t)  =: f_k(X, t),
$$
in the $L^P$-viscosity sense. One easily computes,
$$
   \begin{array}{rcc}
        \|f_k\|_{L^{p, q}(Q_1)}  & = & \displaystyle \rho^{k({2-\xi(n,p,q)})}\rho^{-k\left(\frac{n}{p} + \frac{2}{q}\right)}\left(\int_{\tau-\rho^{2k}}^{\tau} \left(\int_{B_{\rho^{k}}(Y) } |f(Z,s)|^p dZ\right)^{\frac{q}{p}}ds\right)^{\frac{1}{q}}.
    \end{array}
$$
Due to the sharp choice of $\xi(n , p, q) = 2-\left(\frac{n}{p} + \frac{2}{q}\right)$, we have that
$$
    \|f_k\|_{L^{p, q}(Q_1)} = \|f\|_{L^{p, q}(B_{\rho^k}(Y) \times (\tau - \rho^{2k},  \tau])} \leq \|f\|_{L^{p, q}(Q_1)} \leq \eta,
$$
as well as
$$
    \displaystyle \left(\intav{Q_1} \Theta^P_{F_k}(X, t)\right)^{\frac{1}{P}} \leq \left(\intav{Q_1} \Theta^P_F(X, t)\right)^{\frac{1}{P}}  \leq  \eta.
$$
In conclusion, we are allowed to employed Lemma \ref{lemma3.1} to $v_k$, which provides the existence of a universally bounded real number $\overline{\varsigma}_k$ with $|\overline{\varsigma}_k| \leq C$, such that
\begin{equation}\label{eq3.8}
     \displaystyle \sup_{Q_{\rho}} |v_k - \overline{\varsigma}_k| \leq \rho^{\xi(n, p, q)}.
\end{equation}
Finally, if we select
\begin{equation}\label{eq3.9}
     \displaystyle \varsigma_{k+1} \defeq \varsigma_k + \rho^{k \xi(n, p, q)}\overline{\varsigma}_k
\end{equation}
and rescale \eqref{eq3.8} back to the unit picture, we obtain the $(k+1)^{th}$ step in the induction process \eqref{eq3.7}. In addition, we have that
\begin{equation}\label{eq3.10}
     \displaystyle |\varsigma_{k+1} - \varsigma_k|\leq C \rho^{k \xi(n, p, q)},
\end{equation}
and hence the sequence $\{\varsigma_k\}_{k \geq 1}$ is Cauchy, and so it converges. From \eqref{eq3.7} $\varsigma_k \rightarrow u(Y, \tau)$. As well as from \eqref{eq3.10} it follows that
\begin{equation}\label{eq3.11}
     \displaystyle |u(Y, \tau) - \varsigma_k|\leq \frac{C}{1- \rho^{ \xi(n, p, q)}} \rho^{k \xi(n, p, q)},
\end{equation}

Finally, for $0 < r < \rho$, let $k$ the smallest integer such that $(X,t) \in Q_{\rho^k}(Y, \tau)\setminus Q_{\rho^{k+1}}(Y, \tau)$. It follows from \eqref{eq3.7} and \eqref{eq3.11} that
$$
   \begin{array}{ll}
        \displaystyle \sup_{Q_r(Y,\tau)}  \frac{|u(X,t)-u(Y,\tau)|}{d_{\text{par}}((X,t),(Y, \tau))^{\xi(n, p, q)}} & \displaystyle \leq  \sup_{Q_r(Y,\tau)} \frac{|u(X,t)-\varsigma_k| + |u(Y,\tau) - \varsigma_k|}{d_{\text{par}}((X,t),(Y, \tau))^{\xi(n, p, q)}}\\
        & \displaystyle \leq \left(1 + \frac{C}{1- \rho^{\xi(n, p, q)}}\right)\sup_{Q_r(Y,\tau)} \frac{ {\rho}^{k \xi(n, p, q)}}{d_{\text{par}}((X,t),(Y, \tau))^{\xi(n, p, q)}} \\
        & \displaystyle \leq \left(1 + \frac{C}{1- \rho^{\xi(n, p, q)}}\right)\frac{1}{\rho^{\xi(n, p, q)}}.
   \end{array}
$$
The last estimate provides
$$
    \displaystyle \|u\|_{C^{\xi(n, p, q), \frac{\xi(n, p, q)}{2}}(Q_{1/2})} \leq C
$$
and hence the proof of Theorem is concluded.
\end{proof}

\begin{remark} The exponent of H\"{o}lder regularity of our result is sharp. This is can be verify through of following example from \cite{TU}: Let $u \in C_{loc}((-1,0]; L_{loc}^2(B_1)) \cap L_{loc}^2((-1,0];W_{loc}^{1,2}(B_1))$ be a weak solution to
$$
u_t- \Delta u  = f \quad in \quad Q_1
$$
Suppose that $1< \frac{n}{p} + \frac{2}{q}<2$ then for $\alpha \defeq 2-\left(\frac{2}{p}+\frac{n}{q}\right)$ we have that $u \in C_{loc}^{\alpha, \frac{\alpha}{2}}(Q_1)$. Remark that in this case $n_P = \frac{n+2}{2}$.

\end{remark}

\begin{remark} Under VMO assumption of the coefficients of the operator $F$:
$$
	 \displaystyle \intav{Q_r} \Theta^P_F(X, t)  = \text{o}(1),
$$
as $r\to 0$, Theorem \ref{thm3.2} holds without the smallness oscillation condition, as it can always be assumed upon an appropriate scaling.
\end{remark}

\begin{remark}

Under no assumptions on the coefficients, rather than ellipticity, adjustments in the proof of previous Theorem yields
$C_{loc}^{\alpha, \frac{\alpha}{2}}(Q_1)$
where $\alpha \defeq \min\left\{\beta^{-}, 2-\left(\frac{n}{p} + \frac{2}{q}\right)\right\}$ where  $0<\beta<1$ is the maximal exponent from Preposition \ref{CompSol}.
\end{remark}


\section{Parabolic Log-Lipschitz type estimates}\label{LLR}

\hspace{0.6cm} In this section we address the question of finding the optimal and universal modulus of
continuity for solutions of uniformly parabolic equations of the form \eqref{Meq} whose right hand side lies in the borderline space $L^{p,q}(Q_1)$, when $p$ and $q$ lie on the critical surface:
$$
	\frac{n}{p} + \frac{2}{q}=1.
$$
Such estimate is particularly important to the general theory of fully nonlinear parabolic equations. Through a simple analysis one verifies that solutions of \eqref{Meq}, with sources under the above borderline integrability condition should be asymptotically Lipschitz continuous. Indeed, as $\frac{n}{p} + \frac{2}{q} \to 1^{+}$, solutions are parabolically H\"older continuous for every exponent $0< \alpha < 1$. The key goal in this section is to obtain the sharp, quantitative modulus of continuity for $u$.

\begin{lemma}\label{lemma4.1} Let $u \in C^{0}(Q_1)$ be a normalized $L^P$-viscosity solution to \eqref{Meq}. There exist $\eta(\Lambda, \lambda, n) > 0$ and $0< \rho(\Lambda, \lambda, n) \ll \frac{1}{2}$, such that if

\begin{equation}\label{eq4.2}
    \displaystyle \max\left\{\left(\intav{Q_1} \Theta^P_F(X, t)\right)^{\frac{1}{P}},  \|f\|_{L^{p, q}(Q_1)}, \upsilon\right\} \leq \eta
\end{equation}
under the condition $\frac{n}{p} + \frac{2}{q}=1$, then, we can find an affine function $L(X,t) \defeq A+ \langle B, X \rangle$, with universally bounded coefficients, $|A|+|B| \leq C(\lambda, \Lambda, n)$, such that
\begin{equation}\label{eq4.3}
    \displaystyle \sup_{Q_{\rho}} |(u - L)(X,t)| \leq \rho.
\end{equation}
\end{lemma}

\begin{proof} For a $\delta > 0$ which will be chosen \textit{a posteriori}, we apply Lemma \ref{Lemma1} and find a function $h\colon Q_{\frac{1}{2}} \rightarrow \mathbb{R}$ satisfying
$$
h_t - \mathfrak{F}(D^2h) = 0 \quad \mbox{in} \quad Q_{\frac{1}{2}},
$$
in the $L^P$-viscosity sense such that
\begin{equation}\label{eq4.4}
    \displaystyle \sup_{Q_{1/2}} |(u-h)(X, t)|\leq \delta.
\end{equation}
We now define
\begin{equation}\label{eq4.5}
    L(X,t) = h(0,0) + \langle D h(0,0), X\rangle,
\end{equation}
and apply the regularity theory available for $h$,  see for instance \cite{CKS} or \cite{Wang 2}, as to assure the existence of a universal constants $0<\alpha_F<1$ and $C > 0$ such that
\begin{equation}\label{eq4.6}
    \displaystyle |(h - L)(X, t)|\leq Cd_{\text{par}}((X,t), (0,0))^{1+\alpha_F}, \quad \forall  \ 0< |X|^2+|t| < \frac{1}{3}.
\end{equation}
It is time to make universal choices: we set
\begin{equation}\label{eq4.7}
    \displaystyle \rho \defeq \left(\frac{1}{2C}\right)^{\frac{1}{\alpha_F}}< \frac{1}{2} \quad \mbox{and} \quad \delta \defeq \frac{1}{2} \rho,
\end{equation}
which decides the value of $\eta(\Lambda, \lambda, n)>0$ through  the approximation Lemma \ref{Lemma1}. It the sequel we estimate
$$
    \displaystyle \sup_{(X, t) \in Q_{\rho}} |(u-L)(X, t)| \leq \sup_{(X, t) \in Q_{\rho}} |(u-h)(X, t)| + \sup_{(X, t) \in Q_{\rho}} |(h-L)(X, t)| \leq \rho,
$$
and the proof is complete.
\end{proof}


\begin{theorem}\label{thm4.2} Let $u \in C^0(Q_1)$ be an $L^P$-viscosity solution to \eqref{Meq}. There exists a universal constant $\theta_0 >0$ such that if
$$
	\displaystyle \sup_{(Y, \tau) \in Q_{1/2}}  \left(\intav{Q_1} \Theta^P_F(Y, \tau, X, t)\right)^{\frac{1}{P}}   \leq \theta_0,
$$
then, for a universal constant $C>0$ and any $(X,t), (Y, \tau) \in Q_{1/2}$, there holds

$$
    |u(X, t) - u(Y, \tau)| \leq C\{\|u\|_{L^{\infty}(Q_1)} + \|f\|_{L^{p,q}(Q_1)}\}.\omega(d_{\text{par}}((X, t),(Y, \tau))),
$$
where $\omega(s) \defeq s \log \frac{1}{s}$ is the Lipschitz logarithmical modulus of continuity.
\end{theorem}


\begin{proof} We start off the proof by assuming, with no loss of generality, that    $|u|\leq 1$ and
$$
\|f\|_{L^{p, q}(Q_1)} < \frac{\eta}{4} \quad \mbox{and} \quad \upsilon < \frac{\eta}{8\max\{1, \mathcal{L}^n(B_1(0))\}},
$$
where $\eta = \eta(n, \lambda, \Lambda)$ the largest positive number such that the Lemma \ref{lemma4.1} holds. Choose $\theta_0 = \frac{\eta}{8}$. For a fixed $(Y, \tau) \in Q_{1/2}$ we will prove the existence of a sequence of affine functions
$$
    L_k(X, t) = A_k + \langle B_k, X - Y\rangle
$$
such that
\begin{equation}\label{eq4.8}
    \displaystyle \sup_{B_{\rho^k}(Y) \times (\tau -\rho^{2k}  , \tau]} |(u- L_k)(X, t)| \leq \rho^k
\end{equation}
and
\begin{equation}\label{eq4.9}
     \displaystyle |A_{k+1} - A_k| \leq C \rho^k \quad \mbox{and} \quad | B_{k+1} - B_k| \leq C,
\end{equation}
where $0<\rho \ll \frac{1}{2}$ is the radius given by Lemma \ref{lemma4.1}. Notice that the second estimate in \eqref{eq4.9} gives the growing estimate  on the linear coefficients of order
\begin{equation}\label{eq4.91}
     \displaystyle  | B_{k}| \leq Ck.
\end{equation}
We now argue by induction. Lemma \ref{lemma4.1} provides the first step and now we suppose that we have already verified the $k$th step of \eqref{eq4.8}. Define
$$
    v_k(X, t) \defeq \frac{(u-L_k)(Y+\rho^k X, \tau + \rho^{2k} t)}{\rho^k},
$$
which verifies $|v_k|\leq 1$ in $Q_1$, by the induction condition. Define
$$
    \displaystyle F_k(M, \overrightarrow{p}, X, t) := \rho^kF\left(\frac{M}{\rho^k}, \overrightarrow{p}, Y+ \rho^kX,\tau+ \rho^{2k}t\right).
$$
It is plain to check that $F_k$ is a $(\lambda, \Lambda, \upsilon)$-parabolic operator and

$$
    (v_k)_t - F_k(D^2v_k, D v_k, X, t) = f_k(X, t) + g_k(X, t) = H_k(X, t)
$$
in the $L^P$-viscosity sense, where
$$
f_k(X, t)\defeq \rho^kf( Y+\rho^k X, \tau + \rho^{2k} t)
$$
and
$$
g_k(X, t) \defeq \rho^k[F_k(D^2v_k, D v_k + B_k, X, t)-F_k(D^2v_k, D v_k, X, t)].
$$
Moreover,
$$
\begin{array}{cll}
      \|f_k\|_{L^{p, q}(Q_1)}  & = & \displaystyle  \rho^{k}\rho^{-k\left(\frac{n}{p} + \frac{2}{q}\right)}\left(\int_{\tau-\rho^{2k}}^{\tau} \left(\int_{B_{\rho^{k}}(Y) } |f(Z,s)|^p dZ\right)^{\frac{q}{p}}ds\right)^{\frac{1}{q}}.
\end{array}
$$
By the critical condition, $\frac{n}{p} + \frac{2}{q} = 1$, we verify that
$$
    \|f_k\|_{L^{p, q}(Q_1)} = \|f\|_{L^{p, q}(B_{\rho^k}(Y) \times (\tau - \rho^{2k},  \tau])}< \frac{\eta}{4}.
$$
Moreover, given the smallest regime on $\upsilon$, assumption \eqref{eq1} and \eqref{eq4.91}, we have
$$
     |g_k(X, t)|\leq  C k\rho^{k} \upsilon < \frac{\eta}{8\max\{1, \mathcal{L}^n(B_1(0))\}}.
$$
Thus,
$$
    \|g_k\|_{L^{p, q}(Q_1)} \leq \frac{\eta}{8\max\{1, \mathcal{L}^n(B_1(0))\}} \mathcal{L}^n(B_1(0))^{1/p}\leq \frac{\eta}{8}.
$$
Therefore,  $\|H_k\|_{L^{p, q}(Q_1)} < \frac{3\eta}{8} $ . Furthermore,
$$
    \displaystyle \left(\intav{Q_1} \Theta^P_{F_k}(X, t)\right)^{\frac{1}{P}} \leq \left(\intav{Q_1} \Theta^P_F(X, t)\right)^{\frac{1}{P}} \leq  \frac{\eta}{8}.
$$
We have verified that we can apply Lemma \ref{lemma4.1} to the function $v_k$, assuring the existence of an affine function $\widetilde{L}_k(X, t) = \widetilde{A}_k + \langle \widetilde{B}_k, X \rangle$  satisfying $|\widetilde{A}_k| , |\widetilde{B}_k| \leq C$, such that
\begin{equation}\label{eq4.10}
     \displaystyle \sup_{Q_{\rho}} |(v_k - \widetilde{L}_k)(X, t)| \leq \rho.
\end{equation}
We now define
\begin{equation}\label{eq4.11}
      \displaystyle A_{k+1} \defeq A_k + \rho^k \widetilde{A}_k  \quad \text{and} \quad B_{k+1} \defeq B_k + \widetilde{B}_k.
\end{equation}
Rescaling \eqref{eq4.10} to the unit picture gives the $(k+1)$th induction step. The first estimate in \eqref{eq4.8}  assures that the sequence $\{A_k\}_{k \geq 1}$ converges to $u(Y, \tau)$. Also we can estimate, by geometric series,
\begin{equation}\label{eq4.12}
    \displaystyle |u(Y, \tau) - A_k| \leq \frac{C \rho^k}{1-\rho}.
\end{equation}
Finally, for $0 < r < \rho$, let $k$ be the lowest integer such that
$$
    (X,t) \in Q_{\rho^k}(Y, \tau)\setminus Q_{\rho^{k+1}}(Y, \tau).
$$
It follows by \eqref{eq4.8}, \eqref{eq4.91} and \eqref{eq4.12} that
$$
\begin{array}{ll}
    \displaystyle \sup_{Q_r(Y, \tau)}  \frac{|u(X,t)-u(Y,\tau)|}{r\log r^{-1}} & \displaystyle \leq  \sup_{Q_{{\rho}^k}(Y, \tau)} \frac{|(u-L_k)(X,t)| + |u(Y,\tau) - A_k| + |B_k|\rho^k}{r\log r^{-1}}\\
    & \displaystyle \leq C \sup_{Q_{{\rho}^k}(Y, \tau)} \frac{k{\rho}^k}{r\log r^{-1}} \\
    & \displaystyle \leq C,
\end{array}
$$
and the proof of Theorem is conclude.
\end{proof}

\begin{remark}\label{remLL} As a consequence of the estimate given by Theorem \ref{thm4.2}, we are able to derive a precise integral behavior of the gradient of a solution to \eqref{Meq}. Indeed, one can derive the following pointwise control, say near $(0,0)$:
$$
   |D u(X,t)| \lesssim -C\log(|X|^2+|t|) \quad \forall |X|+|t| \ll \frac{1}{2}
$$
Under suitable smallness regime on $f \in L^{p,q}(Q_1)$ and on $\Theta_F \in L^{P}(Q_1)$, it follows by an adjustment of our arguments, combined with $W^{2, 1,s}$ interior estimates from \cite{CKS} that one can approximate an $L^P$-viscosity solution of \ref{Meq} by an $F$-caloric function
$$
    h_t-\mathfrak{F}(D^2h, X, t) = 0 \quad \mbox{in} \quad Q_{1/2},
$$
in the $W^{2, 1, s}(Q_{1/2})$ topology. Thus, through an iterative process as indicated in the proof of Theorem \ref{thm4.2}, one can find affine functions $L_k$ such that
$$
    \displaystyle \intav{Q_{\rho^k}} |D(u-L_k)(X, t)|^sdXdt \leq 1.
$$
Therefore, it is possible to establish BMO type of estimates for the gradient in terms of the $L^{p, q}(Q_1)$ norm of $f$, when the critical condition $\frac{n}{p} + \frac{2}{q}=1$ is verified. That is,
$$
    \|D u\|_{\textrm{BMO}(Q_r)} \leq C[\|u\|_{L^{\infty}(Q_1)}+ \|f\|_{L^{p, q}(Q_1)}], \quad \mbox{for} \quad 0<r\ll\frac{1}{4}.
$$
Comparing such an estimate with the results from \cite{CKS}, it synthesizes quantitatively the fact of
$$
	\displaystyle|D u| \in \bigcap_{p \geq 1} L_{loc}^p(Q_1).
$$
since $L^P$-viscosity solutions have its gradient in $L_{loc}^s(Q_1)$ for all $s< \frac{n+2}{\frac{n}{p}+\frac{2}{q}-1}$.
\end{remark}


\section{Optimal $C^{1+ \alpha, \frac{1+ \alpha}{2}}$ regularity}\label{DR}

\hspace{0.65cm} In this section we obtain asymptotically sharp $C^{1+ \sigma, \frac{1+ \sigma}{2}}$ interior regularity estimates for solutions of \eqref{Meq}. Such estimates are already available in the literature, see for instance \cite{CKS} and \cite{Wang 1}. We shall only comment on how we can deliver them by means of the arguments designed in Section \ref{LLR}.

Initially, 
%
we revisit Lemma \ref{lemma4.1} and observe that if $0<\alpha_F\leq 1$ represents the optimal exponent from the $C^{1+ \overline{\alpha}, \frac{1+ \overline{\alpha}}{2}}$ regularity theory for solutions to homogeneous $(\lambda, \Lambda, \upsilon)$-parabolic operators with constant coefficients, then given
\begin{equation}\label{Choieq}
  \alpha \in (0, {\alpha}_F) \cap \left(0, 1-\left(\frac{n}{p} + \frac{2}{q}\right)\right],
\end{equation}
since $\displaystyle \left(\intav{Q_1} \Theta^{P}_F(X, t)\right)^{\frac{1}{P}}$  and $\|f\|_{L^{p,q}}$ are under universal smallest regime assumption, we are able to choose
\begin{equation}\label{Radius}
   \rho \defeq \left(\frac{1}{2C}\right)^{\frac{1}{{\alpha}_F - \alpha}}
\end{equation}
such that
\begin{equation}\label{eq5.2}
   \displaystyle \sup_{Q_{\rho}} |(u-L)(X, t)|\leq \rho^{1+ \alpha},
\end{equation}
where $L$ is given by \eqref{eq4.5}. This is the first step in our induction process. Now, verified the $k$th step in the induction process
\begin{equation}\label{eq5.2 - ind process}
   \displaystyle \sup_{Q_{\rho^k}} |(u-L_k)(X, t)|\leq \rho^{k(1+ \alpha)}
\end{equation}
with the following order of approximation for the coefficients
\begin{equation}\label{decaycoeff}
|A_{k+1}-A_k|\leq C\rho^{k(1+\alpha)} \quad \mbox{and} \quad |B_{k+1}-B_k|\leq C\rho^{k\alpha}.
\end{equation}
We define the re-scaled function
$$
    v_k(X, t) \defeq \frac{(u-L_k)(Y+\rho^k X, \tau + \rho^{2k} t)}{\rho^{k(1+\alpha)}},
$$
which verifies $|v_k|\leq 1$ in $Q_1$, and satisfies in the $L^P$-viscosity sense
\begin{equation}\label{eq5.4}
  \partial_t v_k- G_k(D^2v_k, Dv_k, X, t) = f_k(X, t) + g_k(X, t) = H_k(X, t)
\end{equation}
where

$$
    \displaystyle G_k(M, \overrightarrow{p}, X, t) = {\rho}^{k(1-\alpha)}F \left( \frac{1}{{\rho}^{k(1-\alpha)}}M, \rho^{k\alpha}\overrightarrow{p}, {\rho^k} X, {\rho}^{2k} t\right)
$$
is a $(\lambda, \Lambda, \upsilon)$-parabolic operator and

$$
f_k(X, t)\defeq \rho^{k(1- \alpha)}f( Y+\rho^k X, \tau + \rho^{2k} t)
$$

$$
g_k(X, t) \defeq \rho^{k(1-\alpha)}[G_k(D^2v_k, D v_k + B_k, X, t)-G_k(D^2v_k, D v_k, X, t)]
$$

Now,
$$
      \|f_k\|_{L^{p, q}(Q_1)} = \omega(\rho^{k})\|f\|_{L^{p, q}(Q_{\rho^k}(Y, \tau))}< \frac{\eta}{2},
$$
where $\omega(\rho^{k}) = \rho^{k\left[1-\alpha-\left(\frac{n}{p}+\frac{2}{q}\right)\right]}$ is computed by change of variables. By the integrability relation and the value of $\alpha$, we conclude $\omega(\rho^k) \leq 1$  for all integer $k\geq 1$.  Also
$$
\begin{array}{ccc}
   |g_k(X, t)| & \leq & \tau(\rho^k)\upsilon\\
\end{array}
$$
where as before $\tau$ is easily computed explicitly using \eqref{decaycoeff}. Again we verify $\tau(\rho^k) < 1$ and in fact  $\displaystyle \lim_{k \to \infty} \tau(\rho^k) = 0$. Thus,
$$
\|g_k\|_{L^{p, q}(Q_1)} \leq \upsilon \mathcal{L}^n(B_1(0))^{1/p}<\frac{\eta}{2}.
$$
Finally,
$$
    \displaystyle \left(\intav{Q_1} \Theta^{P}_G(X, t)\right)^{\frac{1}{P}}\leq \displaystyle \left(\intav{Q_1} \Theta^{P}_F(X, t)\right)^{\frac{1}{P}} \quad \mbox{and} \quad \|H_k\|_{L^{p,q}(Q_1)} \leq  \|f_k\|_{L^{p,q}(Q_1)} + \|g_k\|_{L^{p,q}(Q_1)} < \eta;
$$
therefore, we can apply the first induction step, which gives the existence of an affine function $\overline{L}_k(X) \defeq \overline{A}_k + \langle \overline{B}_k, X  \rangle$ with $|\overline{A}_k|, |\overline{B}_k| \leq C(n, \lambda, \Lambda)$ such that
$$
\displaystyle \sup_{Q_{\rho}}|v_k-\overline{L}_k| \leq \rho^{1+\alpha}.
$$
Rewriting the previous estimate in the unit picture gives
$$
	\displaystyle \sup_{Q_{\rho^{k+1}}}|u-\overline{L}_{k+1}| \leq \rho^{(k+1)(1+\alpha)},
$$
for $L_{k+1}(x) \defeq L_{k}(x)+ \rho^{k(1+\alpha)}\overline{L}_{k}(\rho^{-k}X)$.
The coefficients fulfils
\begin{equation}\label{eq5.5}
     |A_{k+1}-A_k| + {\rho}^k|B_{k+1} - B_{k}| \leq C_0(n, \lambda, \Lambda)
      {\rho}^{(1+\alpha)k},
\end{equation}
hence, from \eqref{eq5.5}, we conclude that  $(A_k)_{k \geq 1} \subset \mathbb{R}$ and
$(B_k)_{k \geq 1} \subset \mathbb{R}^n$ converge  to $u(Y, \tau)$ and to $D u(Y,\tau)$ respectively. Moreover we have the following control
\begin{equation}\label{eq5.7}
   |u(Y, \tau)-A_k| \leq C_0 \frac{{\rho}^{k(1+\alpha)}}{1-\rho} \quad \mbox{and} \quad
   |D u(Y, \tau)-B_k| \leq C_0 \frac{{\rho}^{k\alpha}}{1-\rho}
\end{equation}
Finally, given any $0 < r < \rho$,  let $k$ be an integer such that $ (X,t) \in Q_{\rho^k}(Y, \tau)\setminus Q_{\rho^{k+1}}(Y, \tau)$. Therefore, we estimate from   \eqref{eq5.7} that
$$
     \displaystyle \sup_{Q_r(Y, \tau)} |u(X,t)-[u(Y, \tau)+\langle D u(Y, \tau),
      X-Y\rangle]| \leq r^{1+\alpha}
$$
and the sketch is finished.

\begin{remark} We highlight that the previous result must be interpreted in following way
$$
 \left\{
\begin{array}{cll}
 \hbox{If} \quad \sigma  = 1-\left(\frac{n}{p} + \frac{2}{q}\right)< \alpha_F & \hbox{then} & u \in C_{loc}^{1+\sigma, \frac{1+\sigma}{2}}(Q_1)\\
  \hbox{If} \quad  1-\left(\frac{n}{p} + \frac{2}{q}\right) \geq \alpha_F & \hbox{then} & u \in C_{loc}^{1 + \gamma, \frac{1 + \gamma}{2}}(Q_1), \hbox{ for any } \gamma< \alpha_F.
\end{array}\right.
$$
\end{remark}

\begin{remark} The optimality of previous result can be verified by an example due to Krylov in \cite[Page 209]{KRYL4}.
\end{remark}


\section{Parabolic $C^{1, \mathrm{Log-Lip}}$ type estimates}\label{DLLR}

\hspace{0.6cm}In this last section we  address the issue of finding the optimal regularity estimate  for the limiting upper borderline case $f \in $ BMO, which encompasses the case $f\in L^{\infty, \infty}\simeq L^{\infty}$.

In view of the almost optimal estimates given in the previous section, establishing a quantitative regularity result for solutions to \eqref{Meq} with bounded forcing term, requires that $F$-harmonic functions are $C^{2+ \sigma, \frac{2+ \sigma}{2}}$ smooth; otherwise no further information could be reveled from better hypotheses on the source function $f$. Evans-Krylov's regularity theory \cite{Evans}, \cite{KRYL1} and \cite{KRYL2} assures that convex/concave equations do satisfy the $C^{2+ \sigma, \frac{2+ \sigma}{2}}$ smoothness assumption.

We now state and prove our sharp $C^{1, \text{Log-Lip}}$ interior regularity theorem. For simplicity we will work on equations with constant coefficients and with no gradient dependence. Similar result can be easily obtained under continuity condition on the coefficients and Lipschitz control on the gradient dependence.

\begin{theorem}\label{thm6.1} Let $u \in C^{0}(Q_1)$ be a $C^0$-viscosity solution to $u_t - F(D^2u) = f(X,t) \quad in \quad Q_1$. If any solution of $v_t - F(D^2v +C) = D$, where $C \in Sym(n)$ and $D \in \mathbb{R}$ are on the surface $F(C) = D$, has interior $C^{2+ \sigma, \frac{2+ \sigma}{2}}$ \textit{ a priori} estimates, i.e.,

\begin{equation}\label{eq6.1}
    \|v\|_{C^{2+ \sigma, \frac{2+ \sigma}{2}}(Q_r)} \leq \frac{\overline{\Phi}}{r^{2+\sigma}}\|v\|_{L^{\infty}(Q_1)}
\end{equation}
for some $\overline{\Phi}(\Lambda, \lambda,n ) > 0$. Then, for a constant $C(\overline{\Phi}, \Lambda, \lambda,n) > 0$, there holds
\begin{equation}\label{eq6.2}
    |u(X,t) - [u(0,0) + \langle D u(0,0),  X\rangle]| \leq C\left\{\|u\|_{L^{\infty}(Q_1)}+ \|f\|_{BMO(Q_1)}\right\}\omega(d_{\text{par}}((X,t),(0,0)))
\end{equation}
where $\omega(r) = r^2 \log \frac{1}{r}$ is the $C^1$-Log-Lipschitz  modulus of continuity.
\end{theorem}

\begin{proof} By standard reduction arguments, we can assume that $\|u\|_{L^{\infty}(Q_1)} \leq \frac{1}{2}$ and $\|f\|_{\textrm{BMO}(Q_1)} \leq {\vartheta}_0$ for some ${\vartheta}_0 > 0$ which will be chosen \textit{a posteriori}. Throughout the proof we use the notation
$$
	\displaystyle [f] \defeq \intav{Q_1} f(Z, \varsigma)dZd\varsigma.
$$
The strategy is to find parabolic quadratic polynomials
$$
    \mathcal{P}_k(X,t) \defeq \frac{1}{2}\langle A_kX,X \rangle + B_k t  + \langle C_k, X \rangle + D_k
$$
such that
$ \mathcal{P}_0 = \mathcal{P}_{-1} =  \frac{1}{2}\langle NX,X \rangle$, where $F(N) = [f]$ and for all $k \geq 0$,
\begin{equation}\label{eq6.3}
   B_k - F(A_k)=[f] \quad and \quad \sup_{Q_{\rho^k}} |u- \mathcal{P}_k| \leq \rho^{2k},
\end{equation}
with
\begin{equation}\label{eq6.4}
    \rho^{2(k-1)}(|A_k-A_{k-1}| + |B_k-B_{k-1}|)+ \rho^{k-1}|C_k-C_{k-1}| + |D_k-D_{k-1}| \leq C\rho^{2(k-1)}
\end{equation}
where the radius $0<\rho \ll \frac{1}{2}$ in\eqref{eq6.3} and \eqref{eq6.4} will also be determined \textit{a posteriori}. We prove the existence of such a polynomials by induction process in $k$. The first step of induction, $k = 0$, it is obviously satisfied. Suppose now that we have verified the thesis of induction for $k = 0,1,\ldots, i$. Then, defining the re-scaled function $v\defeq Q_1 \to \R$ given by
$$
    v_i(X,t) = \frac{(u-\mathcal{P}_i)(\rho^i X, \rho^{2i}t )}{\rho^{2i}},
$$
we have, by induction hypothesis, that $|v_i| \leq 1$ and it solves
$$
    \partial_tv_i - F_i(D^2v_i) = f(\rho^i X, \rho^{2i}t)-B_i \defeq f_i(X, t)
$$
in the $C^0$-viscosity sense, where $F_i(M) \defeq F(M+A_i) - B_i$ which is a $(\lambda, \Lambda, 0)$-parabolic operator with
$$
   \begin{array}{lll}
    \|f_i\|_{\textrm{BMO}(Q_1)}  & \defeq&  \displaystyle \sup_{0<r\leq 1} \intav{Q_r} \left|f_i(X,t) - \intav{Q_r} f_i(Y, \tau)dYd \tau\right|dXdt  \\
       &  = &  \displaystyle \sup_{0<r\leq 1} \intav{Q_{r\rho}} \left|f(Z,\varsigma) - \intav{Q_{r\rho}} f(W, \omega)dWd\omega\right|dZd\varsigma\\
   & \leq&  \|f\|_{\textrm{BMO}(Q_1)} \\
&\leq&  \vartheta_0.
   \end{array}
$$
As in Lemma \eqref{Lemma1}, with some slight changes, and, under smallness assumption on $ \|f\|_{BMO(Q_1)}$
to be set soon, we can find a $C^0$-viscosity solution $h$ to
$$
   h_t - F(D^2h + M_i)= [f] \quad in \quad Q_1,
$$
such that
$$
   \sup_{Q_{1/2}} |(v-h)(X,t)|\leq \delta,
$$
for some $\delta>0$ which we will choose below. From hypothesis \eqref{eq6.1}, $h$ is $C^{2+\sigma, \frac{2+\sigma}{2}}$ at the origin with universal bounds. Thus, if we define
$$
    \mathcal{P}(X,t) \defeq \frac{1}{2}\langle D^2h(0,0)X,X \rangle + h_t(0,0)t + \langle D h(0,0),X\rangle + h(0,0),
$$
by the $C^{2+\sigma, \frac{2+\sigma}{2}}$ regularity assumption \eqref{eq6.1}, we can estimate

$$
    |D^2h(0,0)|+ |h_t(0,0)|+ |D h(0,0)|+ |h(0,0)|\leq C\overline{\Phi}
$$
where
$$
    |(h-\mathcal{P})(X,t)| \leq C(n)\overline{\Phi} d_{\text{par}}((X,t),(0,0))^{2+\sigma}.
$$
Now, we are able to select
$$
    \rho \defeq \left(\frac{1}{2C\overline{\Phi}}\right)^{\frac{1}{\sigma}} \quad \text{ and } \quad \delta \defeq \frac{1}{2}\rho^2.
$$
The choice above for $\rho(\overline{\Phi}, \sigma , \Lambda, \lambda, n) \ll \frac{1}{2}$  decides the value for $\delta(\overline{\Phi}, \sigma , \Lambda, \lambda, n)>0$ which determines, by Lemma \eqref{Lemma1}, the universal smallness regime given by the constant $\vartheta_0 > 0$. From the previous choices, we readily obtain
\begin{equation}\label{eq6.5}
    \sup_{Q_{\rho}} |v - \mathcal{P}|\leq \rho^2.
\end{equation}
Rewriting \eqref{eq6.5} back to the unit picture yields
\begin{equation}\label{eq6.6}
    \displaystyle \sup_{Q_{\rho^{i+1}}} \left|u(X,t) - \left[\mathcal{P}_i(X,t)+ \rho^{2i}\mathcal{P}\left(\frac{X}{\rho^i}, \frac{t}{\rho^{2i}}\right)\right]\right|\leq \rho^{2(i+1)}.
\end{equation}
Therefore, defining
$$
    \mathcal{P}_{i+1}(X,t) \defeq \mathcal{P}_i(X,t) + \rho^{2i} \mathcal{P}\left(\frac{X}{\rho^i}, \frac{t}{\rho^{2i}}\right),
$$
we verify the $(i+1)^{\underline{th}}$ step of induction and, clearly, the required conditions \eqref{eq6.3} and \eqref{eq6.4} are satisfied.  From  \eqref{eq6.4} we conclude that $D_k \rightarrow u(0,0)$ and $C_k \rightarrow D u(0.0)$, with the following estimates
\begin{equation}\label{eq6.7}
    |u(0,0)-D_k|\leq \frac{C\rho^{2k}}{1-\rho} \quad \text{and} \quad
    |D u(0,0)-C_k| \leq \frac{C\rho^k}{1- \rho}.
\end{equation}
Furthermore, equation \eqref{eq6.4} yields the grow controls:
\begin{equation}\label{eq6.8}
    \displaystyle |A_i| \leq \sum_{j=1}^{k}|A_j - A_{j-1}|\leq Ck \quad \text{ and } \quad |B_k| \leq \sum_{j=1}^{k} |B_j-B_{j-1}|\leq Ck.
\end{equation}
Finally, given any $0 < r < \rho$,  let $k$ be an integer such that
$$
    (X,t) \in Q_{\rho^k}(Y, \tau)\setminus Q_{\rho^{k+1}}(Y, \tau)
$$
We estimate from equations  \eqref{eq6.3}, \eqref{eq6.7} and \eqref{eq6.8},
$$
\begin{array}{rl}
     \displaystyle \sup_{Q_r(0)} |u(X,t)-[u(0,0)+\langle D u(0,0), X\rangle]| \leq & \displaystyle   {\rho}^{2k} + |u(0,0)-D_k|+\rho^k|D u(0,0)-C_k| \\
	+&  \rho^{2k}(|B_k|+|A_k|)\\
    \leq &  C r^2 \log r^{-1},
\end{array}
$$
and the proof of Theorem is finished.
\end{proof}

\begin{remark}The final estimate says that solutions to \eqref{Meq} are asymptotically $C^{2, 1}$ in the parabolic sense. Furthermore, adjustments on the previous explanation yield $u_t, D^2u \in BMO(Q_{1/2})$, with appropriate \textit{a priori} estimate in terms of the $BMO$-norm of $f$ in $Q_1$. Indeed,  under appropriate smallness regime on $f \in BMO(Q_1)$ we can approximate $u$ by a solution $h$ to
$$
    h_t-\mathfrak{F}(D^2h, X, t) = [f] \quad \mbox{in} \quad Q_{1/2}
$$
in the $W^{2, 1, s}(Q_{1/2})$ topology. Thus, by an iterative process similar to the one used here one finds parabolic quadratic polynomials $\mathcal{P}_k$ such that

$$
    \displaystyle \intav{Q_{\rho^k}}(|\partial_t(u-\mathcal{P}_k)|^s + |D^2(u-\mathcal{P}_k)|^s) \leq 1
$$

Therefore, the previous sentence provides the aimed BMO  estimate. In other words,
$$
   \|u_t\|_{BMO(Q_r)}+ \|D^2 u\|_{BMO(Q_r)} \leq C\{\|u\|_{L^{\infty}(Q_1)} + \|f\|_{BMO(Q_1)}\}, \quad \mbox{for} \quad 0<r\ll 1
$$
\end{remark}


\begin{remark} The result proven  in this Section can be further applied to equations of the form $u_t-F(D^2u, X, t) = f(u, X, t)$, where $f$ is continuous. It is particularly meaningful to geometric flow problems:
$$
  H_t-\Delta H- H|A|^2 = 0,
$$
where $H$ is the inwards mean curvature vector of the surface at position $X$ and time $t$ and $|A|$ represents the norm of the second fundamental form. This equation describes the mean curvature hypersurface in the Euclidean space $\R^{n+1}$, see for example \cite{SW}.
\end{remark}


\begin{remark} As a final remark, we note that the results proven in this article can be generalized for a more general class of anisotropic Lebesgue spaces with mixed norms. Namely, consider $\overrightarrow{p} = (p_1, \ldots, p_n)$. Let $f\in L^{p_1, \ldots , p_n, q}(Q_1)$, i.e., $f \in L_{X_1}^{p_1} \ldots L_{X_n}^{p_n}L_{t}^{q}$.
The quantity 
$$
	\displaystyle \kappa(n, p_1, \ldots, p_n, q) = \left( \sum_{i=1}^{n} \frac{1}{p_i}\right)+ \frac{2}{q}
$$ 
sets up the following regularity regimes, with universal {\it a priori} estimates:
\begin{itemize}
\item  $1<\kappa(n, p_1, \ldots, p_n, q)< \frac{n+2}{n_P}<2$ for the $C^{\alpha, \frac{\alpha}{2}}$ regularity regime;
\item $\kappa(n, p_1, \ldots, p_n, q) =1$ for the Lipschitz logarithmical type estimates;
\item  $0<\kappa(n, p_1, \ldots, p_n, q) <1$ for the $C^{1+\alpha, \frac{1+\alpha}{2}}$ regularity regime.
\end{itemize}
\end{remark}

\bigskip


\end{document}